\documentclass[a4paper,11pt]{scrartcl}
\usepackage[T1]{fontenc}
\usepackage[utf8]{inputenc}
\usepackage{lmodern}
\usepackage{amsmath,amssymb,amsthm}
\usepackage{url}
\usepackage{comment}

\title{Closed-form expansions for the bivariate chromatic polynomial of paths and cycles}
\author{Klaus Dohmen\\ Hochschule Mittweida\\ Technikumplatz 17\\ 09648 Mittweida, Germany}

\newtheorem{theorem}{Theorem}
\newtheorem{lemma}[theorem]{Lemma}

\theoremstyle{remark}
\newtheorem{remark}[theorem]{Remark}

\newcommand{\xy}{\sqrt{(x+1)^2-4y}}

\begin{document}

\maketitle

\begin{abstract}
\textit{Abstract}. 
We establish closed-form expansions for 
the number of colorings of a path or cycle on $n$ vertices with colors from 
the set $\{1,\dots,x\}$ 
such that adjacent vertices are colored differently
or with colors from $\{y+1,\dots,x\}$.
\end{abstract}

\section{Introduction}

Let $G=(V,E)$ be a finite, simple graph having vertex-set $V$ and edge-set $E$.
For any $x\in\mathbb{N}$ and $y = 0,\dots,x$ we use
$P(G,x,y)$ to denote the number of vertex-colorings $f:V\rightarrow \{1,\dots,x\}$
such that for any edge $\{v,w\}\in E$,
either $f(v)\neq f(w)$ or $f(v)=f(w)>y$.
This function has been introduced in \cite{DPT:2003}
and is known to be a polynomial in the indeterminates $x$ and $y$,
which is now referred to as the \emph{bivariate chromatic polynomial}~\cite{AGM:2010}.
This polynomial generalizes the chromatic polynomial (in the particular case where $x=y$),
the independence polynomial, and the matching polynomial.
\par
Let $P_n$ denote the path resp\@. cycle on $n$ vertices.
By considering the lattice of forbidden colorings, it is proved in \cite{DPT:2003} that
\begin{align}
\label{pnsum}
 P(P_n,x,y) & = \sum_{0 < i + 2j \le n} (-1)^{n-i-j} {i+j \choose i} {n-i-j-1\choose n-i-2j} x^i y^j,\\
\label{cnsum}
 P(C_n,x,y) & = (-1)^n y \,+\, n \!\sum_{0 < i + 2j \le n} \frac{(-1)^{n-i-j}}{i+j} {i+j \choose i} {n-i-j-1\choose n-i-2j} x^i y^j.
\end{align}
In this short note, we give closed-form expansions for the sums in (\ref{pnsum}) and (\ref{cnsum}).
Until now, such a closed-form expansion is only known for stars \cite{DPT:2003}.

\section{The path}

Our first result generalizes the fact that the chromatic polynomial
of any path on $n$ vertices is $x(x-1)^{n-1}$.

\begin{theorem}
\label{mainthm}
For any path $P_n$ on $n$ vertices,
any $x\in \mathbb{N}$ and $y=0,\dots,x$,
except for $x=y=1$, we have
\begin{multline*}
P(P_n,x,y) = \frac{\xy-x-1}{2\xy}\cdot \left( \frac{x-1-\xy}{2} \right)^n \\
 + \, \frac{\xy+x+1}{2\xy}\cdot \left( \frac{x-1+\xy}{2} \right)^n\, .
\end{multline*}
\end{theorem}

The proof of Theorem \ref{mainthm} is based on a recent decomposition formula 
for $P(G,x,y)$ by \textsc{Averbouch} et al\@. \cite{AGM:2010}.
This formula involves three kinds of edge elimination:
\begin{labeling}{}
\item[\qquad $G_{-e}$:] The graph obtained from $G$ by removing the edge $e$.
\item[\qquad $G_{/e}$:] The graph obtained from $G$ by identifying the end points of $e$,
and then,
in the resulting multigraph, replacing each pair of parallel edges by a single edge.
\item[\qquad $G_{\dagger e}$:] The graph obtained from $G$ by removing $e$ 
and all incident vertices.
\end{labeling}
We use $\bullet$ resp\@. $\varnothing$ to denote the simple graph
consisting of only one vertex, respectively the empty graph (which has no vertex).

\begin{lemma}[\cite{AGM:2010}]
\label{mainlemma}
For any finite simple graph $G$ and any edge $e$ of $G$,
the bivariate chromatic polynomial $P(G,x,y)$ satisfies the recurrence relation
\begin{align*}
P(G,x,y) & = P(G_{-e},x,y) - P(G_{/e},x,y) + (x-y) \cdot P(G_{\dagger e},x,y) 
\end{align*}
with initial conditions $P(\bullet,x,y)=x$ and $P(\varnothing,x,y)=1$.
\end{lemma}

We proceed with our proof of Theorem \ref{mainthm}.

\begin{proof}[Proof of Theorem \ref{mainthm}]
Obviously,
the statement holds if $x=1$ and $y=0$.
Hence, we may assume that $x>1$.
By choosing an end edge of $G$,
Lemma \ref{mainlemma} yields the recurrence 
\begin{equation}
\label{recurrence}
 P(P_n,x,y) = (x-1) P(P_{n-1},x,y) + (x-y)P(P_{n-2},x,y) \quad (n\ge 3),
\end{equation}
 with initial conditions 
\begin{equation}
\label{initial}
 P(P_0,x,y) = 1, \quad P(P_1,x,y)=x\, .
\end{equation}
Since this is a homogeneous linear recurrence of degree two 
with constant coefficients, its solution is of the form
\begin{equation}
\label{aneq}
 a_n = c_1r_1^n + c_2r_2^n
\end{equation}
where 
\begin{equation}
\label{solution}
r_{1/2} =  \frac{x-1}{2} \pm \frac{1}{2} \sqrt{(x+1)^2-4y}
\end{equation}
are the roots of
the characteristic equation
\begin{equation*}
 r^2 - (x-1) r - x + y = 0 .
\end{equation*}
Note that, since $x>1$ and $y\le x$, 
the discriminant $(x+1)^2-4y$ of the characteristic equation is positive,
so there are exactly two different solutions.
\par
From the initial conditions (\ref{initial}) we obtain 
the following expressions for the coefficients in (\ref{aneq}):
\begin{equation}
 \label{constants}
 c_1 = \frac{\sqrt{(x+1)^2-4y}-x-1}{2\sqrt{(x+1)^2-4y}}, \quad c_2 = \frac{\sqrt{(x+1)^2-4y}+x+1}{2\sqrt{(x+1)^2-4y}}.
\end{equation}
Now, by putting the expressions from (\ref{solution}) and (\ref{constants}) into (\ref{aneq})
the statement of the theorem is proved.
Alternatively, proceed by induction on $n$, using (\ref{recurrence}) and (\ref{initial}).
\end{proof}

\begin{remark}
For $x=y$ the formula in Theorem \ref{mainthm} specializes 
to the chromatic polynomial of $P_n$,
which coincides with the chromatic polynomial of any tree on $n$ vertices.
Note, however, that the formula in Theorem \ref{mainthm} does \emph{not} extent to trees.
As an example, for the star $S_4$ on four vertices one easily finds that
$P(S_4,x,y)=3xy-3x^2y+x^4-y$, 
whereas for the path on four vertices, we have
$P(P_4,x,y) = 2xy-3x^2y+x^4-y+y^2$.
\end{remark}

\section{The cycle}

Based on our preceding result on paths, 
we subsequently generalize the fact that the chromatic polynomial
of any cycle on $n$ vertices is $(x-1)^n+(-1)^n(x-1)$.

\begin{theorem}
For any cycle $C_n$ on $n\ge 3$ vertices,
any $x\in \mathbb{N}$ and $y=0,\dots,x$,
\begin{multline}
\label{cyclepoly}
P(C_n,x,y) = {\left( \frac{x - 1 - \xy }{2} \right)}^{n} \\
+ \, {\left( \frac{x - 1 + \xy}{2} \right)}^{n}
+ \, \left(-1\right)^{n} (y - 1) .
\end{multline}
\end{theorem}

\begin{proof}
Since $P(C_n,1,1)=0$, we may assume that not both $x$ and $y$ are equal to~1.
By choosing an edge of $G$, Lemma \ref{mainlemma} yields the recurrence
\begin{equation}
\label{recurrence_for_cn}
 P(C_n,x,y) + P(C_{n-1},x,y) = P(P_n,x,y)+(x-y)P(P_{n-2},x,y) \quad (n\ge 4),
\end{equation}
with initial condition
\begin{equation*}
P(C_3,x,y) = x^3 -3xy+2y.
\end{equation*}
Iterating (\ref{recurrence_for_cn}) we obtain
\begin{equation}
\label{elf}
P(C_n,x,y) = (-1)^n \left( \sum_{i=4}^n (-1)^i \Big( P(P_i,x,y) + (x-y) P(P_{i-2},x,y) \Big) - P(C_3,x,y) \right).
\end{equation}
Since not both $x$ and $y$ are equal to 1,
we may apply Theorem \ref{mainthm} and write
\begin{equation}
\label{zwoelf}
P(P_i,x,y) + (x-y) P(P_{i-2},x,y) = cr^i +ds^i 
\end{equation}
where 
\begin{gather*}
r = \frac{x-1-\xy}{2}, \quad c = \frac{\xy-x-1}{2\xy}\left( 1+ \frac{x-y}{r^2} \right), \\
s = \frac{x-1+\xy}{2}, \quad d = \frac{\xy+x+1}{2\xy}\left( 1+ \frac{x-y}{s^2} \right).
\end{gather*}
From (\ref{zwoelf}) it follows that
\begin{multline*}
\sum_{i=4}^n (-1)^i \Big( P(P_i,x,y) + (x-y) P(P_{i-2},x,y) \Big) \\
\begin{aligned}
= & \,\,\, c\left( \sum_{i=0}^n (-r)^i + r^3-r^2+r-1 \right) \,+\, d\left( \sum_{i=0}^n (-s)^i +s^3-s^2+s-1 \right) \\
= & \,\,\, c\left( \frac{1-(-r)^{n+1}}{1-(-r)} + (r-1)^3 \right) \,+\, d\left( \frac{1-(-s)^{n+1}}{1-(-s)} + (s-1)^3 \right) .
\end{aligned}
\end{multline*}
Putting this into (\ref{elf}) we obtain 
\begin{equation}
P(C_n,x,y) = c \, \frac{r^{n+1}+(-1)^n r^4}{r+1} \,+\, d \,\frac{s^{n+1}+(-1)^n s^4}{s+1} + (-1)^{n-1} P(C_3,x,y) , 
\end{equation}
which \emph{Sage} \cite{Sage} simplifies to (\ref{cyclepoly}).
\end{proof}


\begin{thebibliography}{9}

\bibitem{AGM:2010}
\textsc{I. Averbouch}, \textsc{B. Godlin}, and \textsc{J.A. Makowsky},
\emph{An extension of the bivariate chromatic polynomial}, 
Europ\@. J\@. Combin\@. \textbf{31} (2010), 1--17. 

\bibitem{DPT:2003}
\textsc{K. Dohmen}, \textsc{A.~P\"onitz}, and \textsc{P.~Tittmann}, 
\emph{A new two-variable generalization of the chromatic polynomial},
Discrete Math\@. Theoret\@. Comput\@. Sci\@. \textbf{6} (2003), 69--90.

\bibitem{Sage}
\textsc{W. Stein} et al., \emph{Sage Mathematics Software (Version 4.7.1)},
The Sage Development Team, 2011, \url{http://www.sagemath.org}.

\end{thebibliography}
\end{document}